\documentclass[12pt]{amsart} 
\usepackage{color}
\usepackage{amssymb}
\usepackage{times}
\usepackage{amsmath}
\usepackage[pdftex]{graphicx}
\usepackage{hyperref}
\usepackage[square,numbers]{natbib}
\usepackage{ mathrsfs }

\newcounter{sideremark}

\newtheorem{Thm}{Theorem}[section]
\newtheorem{Lem}[Thm]{Lemma}

\newtheorem{Cor}[Thm]{Corollary}   
\newtheorem{Def}[Thm]{Definition} 

\newtheorem{remark}[Thm]{Remark}
\newtheorem*{org}{Organization}

\renewenvironment{cases}{\left\{\begin{array}{ll}}{\end{array}\right.}

\newcommand{\vol}{\operatorname{vol}}

\newcommand{\diam}{\operatorname{diam}}
\newcommand{\dist}{\operatorname{dist}}

\renewcommand{\epsilon}{\varepsilon}

\renewcommand{\rho}{\varrho}

\begin{document}
\title[Linear bounds for the lengths of geodesics]{Linear bounds for the lengths of geodesics on manifolds with curvature bounded below}

\author[Beach]{Isabel Beach}
\address[Beach]{Department of Mathematics,  University of Toronto, Toronto, Canada}
\email{isabel.beach@mail.utoronto.ca}

\author[Contreras Peruyero]{Hayde\'e Contreras Peruyero}
\address[Contreras Peruyero]{Centro de Ciencias Matemáticas, UNAM, Morelia, Mexico }
\email{haydeeperuyero@matmor.unam.mx}


\author[Rotman]{Regina Rotman}
\address[Rotman]{Department of Mathematics,  University of Toronto, Toronto, Canada}
\email{rina@math.utoronto.edu}

\author[Searle]{Catherine Searle}
\address[Searle]{Department of Mathematics, Statistics, and Physics, Wichita State University, Wichita, Kansas}
\email{searle@math.wichita.edu}

\date{\today}

\begin{abstract} 
Let $M$ be a simply connected Riemannian manifold in $\mathscr{M}_{k,v}^D(n)$, the  space of closed Riemannian manifolds of dimension $n$ with sectional curvature bounded below by $k$, volume bounded below by $v$, and diameter bounded above by $D$. Let $c$ be the smallest positive real number such that  any closed curve of length at most $2d$ can be contracted to a point over curves of length at most $cd$, where $d$ is the diameter of $M$. In this paper, we show that under these hypotheses there exists a computable rational function, $G(n,k,v,D)$, such that any continuous map of $S^l$ to $\Omega_{p,q}M$, the space of piecewise differentiable curves on $M$ connecting $p$ and $q$, is homotopic to a map 
whose image consists of curves of length at most $\exp(c\exp(G(n,k,v,D))$. In particular, for any points $p,q \in M$ and any integer $m>0$ there exist at least $m$ geodesics connecting $p$ and $q$  of length at most $m\exp(c\exp(G(n,k,v,D))$.
\end{abstract}

\maketitle


\section{Introduction}
Let $M$ be an $n$-dimensional closed Riemannian manifold. In 1951, it was proven by J. P. Serre in \cite{serre} that there exist infinitely many geodesics connecting any pair of points $p, q \in M$. In 1958, A. Schwarz demonstrated that there always exist at least $m$ geodesics between any pair of points on $M$ of length at most $C(M)m$ \cite{schwarz}. It is natural to try to estimate $C(M)$ in terms of various geometric parameters of $M$. Note that when $p=q$, geodesic segments become geodesic loops based at $p$. The first curvature-free estimates for the length of a shortest geodesic loop were obtained by S. Sabourau in \cite{sabourau_2004_loops}. 
\par
Recall that $\mathscr{M}_{k,v}^D(n)$ denotes the  space of closed Riemannian manifolds of dimension $n$ with $\sec(M)\geq k$, $\vol(M)\geq v$, and $d=\diam(M)\leq D$.
In this paper we will estimate $C(M)$ in terms of $k$, $v$, $D$, $n$, 
and a constant $c$, defined below. It is not known whether we can remove any of these constraints in dimensions $n>2$ and still obtain a linear bound. 
The 
conjecture
that one can always bound the $m$th shortest geodesic by $md$ was shown to be false by F.  Balacheff, C.  Croke, and M. Katz in \cite{bck}. 
On the other hand, in dimension two it has been established that there always exist at least $m$ geodesics of length at most $22md$ between any pair of points. The most difficult case is that of a Riemannian $2$-sphere, for which this bound was established by A. Nabutovsky and the fourth author in \cite{rotman_linear_bounds}. It was significantly improved to $8md$ by H. Y. Cheng in \cite{cheng}. In \cite{nr7}, Nabutovsky and the fourth author demonstrated that on any closed, $n$-dimensional Riemannian manifold $M^n$ there exist at least $m$ geodesics of length at most $4nm^2d$ connecting an arbitrary pair of points. While this bound is curvature-free, it is quadratic in $m$. 
\par
When using various curvature bounds, bounds that are linear in $m$ can be established in some specific situations. For example, for a Riemannian manifold $M^n$ of dimension $n$ with Ricci curvature $\mathrm{Ric}(M)\geq (n-1)$,  H.-B. Rademacher demonstrated in \cite{Rad} that for any pair of points $p, q$ there exist at least $m$ geodesics of length at most $((2n-1)m+1)\pi$, improving the prior bound established in \cite{rotman2023}. 
A similar estimate can be established when $M$ is a Riemannian $3$-sphere and $\mathscr{M}_{-k,v, \cdot}^{k, \cdot, D}(n)$-- that is, with $ |\mathrm{sec}(M)| \leq k$,  $\vol(M) \geq v$ and $d=\diam(M) \leq D$. In that case, one can show that there exists a computable function $f_1(k,v,D)$ such that the lengths of the first $m$ shortest geodesics are at most $f_1(k,v,D) m$ (see \cite{nr8}).
In dimensions $n > 3$, the Cheeger--Gromov compactness theorem for $M\in \mathscr{M}_{-k,v, \cdot}^{k, \cdot, D}(n)$ (see ~\cite{stefan_peters}, ~\cite{greene_wu}) 
implies the existence of a function $f_2(k,v, D)$ such that each loop of length at most $2d$ can be contracted to a point via loops based at a fixed point of length at most $f_2(k, v, D)$. Similarly, let $M$ be a simply connected  Riemannian $4$-manifold with $|\mathrm{Ric}(M)| \leq 3$, $\vol \geq v$, $d=\diam(M) \leq D$. N. Wu and Z. Zhu demonstrated in \cite{wuzhu} that there exists a function $f_3(v, d)$ such that any closed curve of length at most $2d$ can be contracted to a point over based point loops of length at most $f_3(v,d)$. The existence of a bound that is linear in $m$ for the length of the shortest $m$ geodesics connecting a pair of points $p,q \in M$ readily follows from the existence of such functions. 
\par
Before we can state our result, we must first establish some notation. 
Given a manifold $M$ and a pair of points $p,q\in M$, we let $\Omega_{p,q} M$ denote the space of piecewise differentiable curves in $M$ connecting $p$ and $q$. We then define $\Omega_{p,q}^{L}M\subset \Omega_{p,q}M$ to be the subspace of such curves of length at most $L\geq 0$. We equip $\Omega_{p,q} M$ with the sup norm, so that $\dist(\alpha,\gamma)=\sup_{t}\dist(\alpha(t),\gamma(t))$ for $\alpha,\gamma\in \Omega_{p,q} M$.
\par
Our goal in this paper is to prove the following theorem.

\begin{Thm} 
    \label{TheoremAA}
    Let $M\in \mathscr{M}_{-1,v}^D(n)$ be simply connected and analytic with $\diam(M)=d$. Let $c>0$ such that any closed curve of length at most $2d$ on $M$ can be homotoped to some point over curves of length at most $cd$. Then given any $\delta>0$ and any continuous map $f: S^l \to \Omega_{p,q} M$, there exists  a rational function $G(n, v, D)$ and a map $g:S^l \to \Omega_{p,q}^{L}M$ homotopic to $f$, where 
      $$L\leq e^{ce^{G(n, v, D)}}+\delta.$$
\end{Thm}
\noindent
In fact, in Lemma \ref{lemma:homotopy_loop_width_bound} we show that $L=l(5W(c,n, v, D) +3D)+D+ \delta$, where $W$ is an explicitly defined function.
We can then extend this result to manifolds with an arbitrary lower sectional curvature bound $k$ as follows. Consider $M\in \mathscr{M}_{k,v}^D(n)$. After scaling the metric on $M$ by $|k|$, we obtain $M'\in \mathscr{M}_{-1,v'}^{D'}(n)$, where $v'=v|k|^n$, and $D'=|k|D$. Moreover, any closed curve of length at most $2d'$ on $M'$ can be homotoped to some point over curves of length at most $cd'$, where $d'=|k|d$. Then $M'$ satisfies Theorem \ref{TheoremAA}, so given any continuous map $f': S^l \to \Omega_{p,q} M'$, there exists a map $g':S^l \to \Omega_{p,q}^{L'}M'$ homotopic to $f'$, where $L'=2l(5W(c, n, v', D') +3D')+D'+ \delta$. Therefore, given any continuous map $f: S^l \to \Omega_{p,q} M$, there exists a map $g:S^l \to \Omega_{p,q}^{L'}M$
homotopic to $f$, as desired. We therefore have the following corollary.
\begin{Cor}
    \label{cor:TheoremAA}
    Let $M\in \mathscr{M}_{k,v}^D(n)$ be  simply connected with $k<0$ and $\diam(M)=d$. Let $c>0$ such that any closed curve of length at most $2d$ on $M$ can be homotoped to some point over curves of length at most $cd$. 
    Then given any $\delta>0$ and any continuous map $f: S^l \to \Omega_{p,q} M$, there exists a rational function $\tilde{G}(n,v, k, D)$ and a map $g:S^l \to \Omega_{p,q}^{L'}M$
    that is homotopic to $f$,
    where 
    \begin{align*}
      L'\leq e^{ce^{\tilde{G}(n, k, v, D)}}+\delta.
    \end{align*}
\end{Cor}
\noindent
Here $\tilde{G}(n,k,v,D)=G(n, v|k|^n, D|k|)$.
\par
Recall that in \cite{serre} the proof of the existence of infinitely many geodesics uses Morse theory on the path space $\Omega_{p,q}M$. Using rational homotopy theory one can show that there exists a non-trivial even-dimensional spherical cohomology class $u$ in the rational cohomology ring of $\Omega_{p,q}M$ with non-trivial cup powers $u^i$. Either these cup powers correspond to different critical points of the energy functional on $\Omega_{p,q}M$ or, if $u^i$ and $u^j$ correspond to the same critical point, Lusternik--Schnirelmann theory establishes the existence of an entire critical level of geodesics.
\par 
In \cite{schwarz}, an effective version of this proof is presented that uses homology classes that are dual under the Pontryagin product, which is generated by concatenation of loops. In particular,  given a homology class $v$ dual to $u$-- that is, satisfying $\langle u,v \rangle=1$-- its Pontryagin powers are dual to the cup powers of $u$ up to some nonzero constant. Thus, either these Pontryagin powers correspond to different geodesics or they provide infinitely many geodesics of the same length.
\par 
With this in mind, Theorem \ref{TheoremAA} gives rise to a length bound on geodesics as follows. Given a simply connected Riemannian manifold $M\in\mathscr{M}_{k,v}^D(n)$, let $f:S^l\to \Omega_{p,q}M$ be a representative of the class $v$ described above. Note that it can be shown that $l\leq 2n-2$. By Corollary \ref{cor:TheoremAA}, $f$ is homotopic to a map whose image consists of curves of length at most 
\begin{align*}
    2l(5W(c, n, v|k|^n, D|k| +3D) + D|k|.
\end{align*}
Therefore for each $m\geq 1$, the Pontryagin power $v^m$ has a representative whose image consists of curves of length at most 
\begin{align*}
    m\left(2l(5W(c, n, v|k|^n, D|k| +3D) + D|k|\right).
\end{align*}
Each $v^m$ gives rise to a geodesic connecting $p$ to $q$ on $M$. We therefore have the following result.
\begin{Cor}
    Let $M$ be a simply connected Riemannian manifold in $\mathscr{M}_{k,v}^D(n)$ for $k<0$. Then any pair of points $p,q\in M$ are connected by $m$ geodesic segments of length at most
    $$me^{ce^{\tilde{G}(n,v, k,D)}},$$
    where $\tilde{G}(n,v, k,D)$ is a computable rational function.
\end{Cor}
\noindent
We note that a similar bound can be established for non-simply connected manifolds $M^n$ by using the same argument as in ~\cite{nr7}, which involves using the universal cover $\widetilde{M}^n$ of $M^n$ with the covering metric. 

\begin{org} 
The paper is organized as follows. In Section \ref{2}, we establish the first step of the proof of Theorem \ref{TheoremAA} and provide a detailed outline of the remaining four steps.
The second step is proven in Section \ref{3} and the third and fourth  in Section \ref{4}. We prove the fifth and final step in Section \ref{5}, by showing that given a sphere in $\Omega_{p,q}M$, which potentially passes through some long paths, we can construct a homotopy between this sphere and a sphere in $\Omega_{p,q}^{2m(5W(c,n, v, D) +3D)+D+ \delta}M$.
\end{org}

\section{Preliminaries and Main Ideas}\label{2}

In this section we give a more detailed outline of the proof of Theorem \ref{TheoremAA}. Our first goal is to state Theorem \ref{lem:grove_petersen} below, which allows us to choose a cover of $M\in \mathscr{M}_{k,v}^D(n)$ by balls of radius strictly less than a certain function $r(n,k,v,D)$, where the number of balls in the cover is bounded by some function of $n, k, v$ and $D$. Before we do so, we first estimate the number of balls in such a cover.
\par
We begin by fixing $\epsilon>0$. Let $N$ be the maximal number of pairwise disjoint balls in $M$ of radius $\epsilon/12$. If such a set of metric balls is given by $\{B(p_i, {\epsilon/12})\}_{i=1}^N$ for some $p_i\in M$, then the set $\{B(p_i, {\epsilon/6})\}_{i=1}^N$ covers $M$. Note that if we remove any of the metric balls from 
$\{B(p_i, {\epsilon/6})\}_{i=1}^N$, it no longer covers $M$. Let $M^n_{\kappa}$ denote the $n$-dimensional space of constant curvature $\kappa$ and let $B_{\kappa}(\bar{p}_i, r)$ denote the corresponding ball of radius $r$ about $\bar{p}_i\in M^n_{\kappa}$. Using the Bishop--Gromov volume comparison theorem \cite{B} (see also Lemma 7.1.4 of \cite{Pet}), we obtain
    $$N\leq \frac{\vol(M)}{\vol(B(p_i,\epsilon/12))} \leq \frac{\vol(B_{-1}(\bar{p}_i, D))}{\vol(B_{-1}(\bar{p}_i, \epsilon/12))}.$$
Thus,
\begin{equation*}
    N \leq \frac{\int_{0}^D\sinh^{n-1}tdt}{\int_0^{\epsilon/12}\sinh^{n-1}t dt}=\phi(\epsilon,D,n)
\end{equation*}
Since $x\leq \sinh(x)\leq e^x/2$ for $x\geq 0$, we have 
\begin{equation}\label{N}
   N\leq \phi(\epsilon,D,n) \leq\frac{\int_{0}^D e^{t(n-1)}dt}{2^{n-1}\int_0^{\epsilon/12}t^{n-1} dt}
    \leq     
    \frac{12^n n e^{D(n-1)}}{2^{n-1}(n-1)\epsilon^{n}}.
\end{equation}
Therefore, for any $\epsilon>0$, there is a cover of $M$ by metric balls of radius $\epsilon/6$ where the number of balls is bounded in terms of $\epsilon, D$ and $n$. 
\par
To establish Theorem \ref{lem:grove_petersen}, we begin by recalling the following definition. 
\begin{Def}[{\bf Width}]\label{width}
    Let $H: X \times [0,1] \to M$ be a homotopy. We define its width to be $W_H=\sup_{x \in X} \ell(H(x,\tau))$, where $L$ is the length functional. In other words, the width of a homotopy is the maximum length of the trajectory of a point of $X$ during the homotopy. 
\end{Def}
\noindent
Let $M$ be a simply connected Riemannian manifold in $\mathscr{M}_{-1,v}^D(n)$. Theorem 1.6 in \cite{gp} allows us to cover $M$ by metric balls that can each be contracted within a larger metric ball of bounded radius. In particular, the authors show that there exist functions $r=r(n,k,v,D)$ and $R=R(n,k,v,D)$ such that there is a differentiable strong deformation retraction $H:\Delta(r)\times [0,1]\to \Delta(r)$ of $\Delta(r)$ onto the diagonal $\Delta\subset M\times M$ with $\ell(H((p,q),\cdot)) \leq R \dist(p,q)$, where 
\begin{align*}
    \Delta(r) = \{(p,q) \mid p,q \in M, \dist(p,q)< r\}.
\end{align*}
\noindent
We restate the theorem here adapted to our purposes, noting that the trajectory of the point $(p,q)$ under this retraction has length $R\dist(p,q)$, and hence the retraction has width at most $Rr$. In particular, given any $p\in M$, $H((p,\cdot),\cdot)$ is a strong deformation retraction of $B(p, r)$ onto $p$ of width $R\dist(p,q)$ for any $q\in B(p, r)$. Thus given any $p\in M$ and $0<\epsilon < r$ the ball $B(p, \epsilon)$ is contractible inside the ball centered at $p$ of radius $R\epsilon$, leading to the following result.

\begin{Thm}
    \cite{gp}
    \label{lem:grove_petersen}
    Suppose $M\in\mathscr{M}^D_{k,v}(n)$. Then there exist functions $r=r(n,k,v,D)$ and $R=R(n,k,v,D)$ such that any ball of radius $\epsilon < r$ in $M$ is contractible inside the concentric ball of radius $R\epsilon$ by a homotopy with width at most $R\epsilon$. 
\end{Thm}

\begin{remark} \label{NR}
    The fourth author shows in the proof of Theorem B in \cite{rotman2000} that for $M \in \mathscr{M}_{-1,v}^D(n)$, one can take 
    \begin{align*}
        r(n,-1,v,D) &= c_1(n) \frac{v\min\{1,v\}}{De^{(n-1)D}}\\
        R(n,-1,v,D) &= c_2(n) \frac{e^{(n-1)D}}{v}
    \end{align*}
    for some constants $c_1(n)$ and $c_2(n)$.
\end{remark}
\par
We may now use Theorem \ref{lem:grove_petersen} with $\epsilon= r(n,k,v,D)/(2a)$ for any constant $a>2$ to produce a cover of $M$ by contractible balls of radius strictly less than $r(n,k,v,D)$, where the bound on the number of balls $N$ is given by Inequality \ref{N}.
\par
Define $\Lambda M$ as the space of all piecewise differentiable closed curves in $M$ parameterized by $[0,1]$ with constant speed. Let $\Lambda^{L}M$ be the space of curves in $\Lambda M$ of length at most $L$. For the second step of the proof of Theorem \ref{TheoremAA}, we establish in Lemma \ref{lemma:net} that given any $\epsilon,L>0$ we can construct an $\epsilon$-net $\Phi$ of $\Lambda^{L}M$ in the larger space $\Lambda^{3L}M$. In this context, an $\epsilon$-net covering a subset $M'\subset M$ is a set of points $\{p_i\}$ in $M$ such that $\{B(p_i, \epsilon)\}$ covers $M'$ (see Definition \ref{e-net}). 
Thus, given any piecewise differentiable curve closed $\gamma$ of length at most $L$, there is a curve $\eta \in \Phi$ of length at most $3L$ such that $\dist(\gamma, \eta) < \epsilon$. Moreover, the number of elements in $\Phi$ will be bounded in terms of $\epsilon, k, v$ and $D$. Recall that the nerve of a cover of $M$ is an abstract simplicial complex approximating $M$, recording the pattern of intersections between sets in the cover. The $\epsilon$-net that we construct here consists of short closed curves in the $1$-skeleton of the nerve of the cover of $M$ obtained using Theorem \ref{lem:grove_petersen}.
\par
For the third step, we consider a homotopy $\gamma_{\tau}, \tau \in [0,1]$, that contracts a closed piecewise differentiable curve $\gamma=\gamma_0$ of length at most $2d$ over curves of length at most $cd$ for some constant $c$. Our goal in Lemma \ref{lemma:homotopy_loop_width_bound} is to approximate this homotopy using curves in our $\epsilon$-net $\Phi$, where we take $\epsilon=r(n,k,v,D)/2a$ and $L=cd$. In order to do so, we first subdivide $[0,1]$ into intervals $0=\tau_0 <\tau_1 <\ldots<\tau_m=1$ so that $\dist(\gamma_{\tau_i}, \gamma_{\tau_{i+1}}) < \epsilon$. We find a closest element in the $\epsilon$-net $\Phi$ to each $\gamma_{\tau_i}$ and denote it by $\eta_i$. If in this sequence we have $\eta_i = \eta_j$ for some $i< j$, we discard the terms $\eta_i,\ldots, \eta_{j-1}$, thus obtaining a new sequence in which all elements are different. Thus, the number of elements in this sequence will be bounded by the number of elements in $\Phi$, and hence by $N$, the number of metric balls in our cover of $M$.
\par 
For the fourth step, using Lemma \ref{lemma:close_curves_homotopy_bounded_width} we construct a homotopy between $\eta_i$ and $\eta_{i+1}$ with width bounded by a function of $n,k,v$ and $D$. This is possible because $\eta_i$ and $\eta_{i+1}$ are close. Then in Lemma \ref{lemma:homotopy_loop_width_bound} we construct a new homotopy between $\gamma(t)$ and a point  by first homotoping $\gamma(t)$ to $\eta_0(t)$, then by piecing together the homotopies between $\eta_i$ and $\eta_{i+1}$ for all $i$, and finally homotoping $\eta_m(t)$ to a point. Finally, in Lemma \ref{lemma:homotopy_final_bounds}, we construct a based-point homotopy contracting $\gamma$ to the point curve $\gamma(0)$ through curves based at $\gamma(0)$ that also have short length.
\par
For the fifth and final step of the proof of Theorem \ref{TheoremAA}, we consider any map $f:S^l\to \Omega_{p,q}M$ and then define a new map $g:S^l\to \Omega_{p,q}^LM$ that is homotopic to the original $h$. To obtain this homotopy, we divide $S^m$ into small $m$-cubes and define the map inductively on the $k$-skeleton. We prove the base case in the proof of Theorem \ref{TheoremAA} by interpolating between points on the 0-skeleton using the homotopies obtained via Lemmas \ref{lemma:homotopy_final_bounds} and \ref{Lemmapathhomotopy}. The proof of the inductive step is accomplished in Lemma \ref{Lemmashortpaths}.

\section{Constructing a net in the space of closed curves of length at most $cd$}\label{3}

Recall that our goal over the next two sections is to establish the following result. Given any piecewise differentiable closed curve of length at most $2d$ that can be contracted to a point via curves of length at most $cd$, for some constant $c$, we show how to contract it to some other point through short loops based at that point. In order to do so we will make use of the existence of a cover of $M$ by a bounded number of contractible balls provided by Theorem \ref{lem:grove_petersen}.
We want to convert this cover to a cover of $\Lambda^L M$, the space of piecewise differentiable closed curves of length at most $L$ parametrized by the unit interval with constant speed, inside the larger space $\Lambda M$ of all piecewise differentiable closed curves parametrized by the unit interval with constant speed. More specifically, we will find a cover by metric balls centered at curves of length at most $3L$. The set of these curves is called a net, whose definition we now recall.
\begin{Def}
    \label{e-net}
    Given any $\epsilon>0$, an $\epsilon$-net covering a susbet $Y$ of a metric space $X$ is a subset of points $\{x_i\}_{i\in I}\subset X$ such that the collection of metric balls $\{B(x_i, \epsilon)\}_{i\in I}$ covers $Y$.
\end{Def}
\noindent 
We describe the procedure in the following lemma, which is adapted from Lemma 3.4 in \cite{rotman2000}. We include the proof for the sake of completeness. This result applies to any cover of a compact manifold by finitely many metric balls and any $L>0$. Later, we will apply this lemma to our special cover and for a prescribed value of $L$.

\begin{Lem}\cite{rotman2000}
    \label{lemma:net}
    Pick $\epsilon>0$ and let $\{B(p_i, \frac{\epsilon}{6} )\}_{i=1}^N$ be a cover of a compact manifold $M$ such that the balls $\{B(p_i, \frac{\epsilon}{12} )\}_{i=1}^N$ are mutually pairwise disjoint. Given $L>0$, there exists an $\epsilon$-net $\Phi$ on $\Lambda^L M\subset \Lambda M$ with $\widetilde{N} \leq N^{\frac{18L}{\epsilon}+1}$ elements. Moreover, the length of every closed curve in $\Phi$ is at most $3L$.
\end{Lem}
\begin{proof}
    Consider all minimizing geodesic segments joining pairs of centers of balls $(p_i, p_j)$ in our cover that satisfy $\dist(p_i,p_j)\leq \frac{\epsilon}{2}$. Note that $\dist(p_i,p_j)\geq \epsilon/6$ for $i\not=j$ by assumption. Let $\Phi$ be the collection of closed curves of total length at most $3L$ formed by connecting such geodesic segments. Note that the length restriction on $\Phi$ means that
    each such curve consists of at most $18L/\epsilon$ geodesic segments. We claim that the set $\Phi$ is the desired $\epsilon$-net.
    \par 
    Given $\gamma\in\Lambda^L M$, we need to find an element in $\Phi$ within distance $\epsilon$ of $\gamma$. Pick $m\leq 18L/\epsilon$ and subdivide the interval $[0,1]$ by $0=t_0<t_1<\ldots<t_{m-1}<t_{m}=1$ so that $\ell(\gamma|_{[t_i,t_{i+1}]})<\epsilon/6$. Let $q_{i}$ be a point in the set $\{p_j\}_{j=1}^N$ that is closest to $\gamma(t_i)$, with $q_0=q_{m}$. Because $\{B(p_i, \frac{\epsilon}{6})\}_{i=1}^N$ covers $M$, $\dist(q_i,\gamma(t_i))\leq \epsilon/6$ and hence 
    \begin{align*}
        \dist(q_i,q_{i+1})
        &\leq \dist(q_i,\gamma(t_i))+\dist(\gamma(t_i),\gamma(t_{i+1}))+\dist(q_{i+1},\gamma(t_{i+1}))\\
        &\leq \epsilon/2.
    \end{align*}
    Let $\sigma_i$ be a minimizing geodesic segment joining $q_i$ and $q_{i+1}$ (see Figure \ref{fig:net}). Because $m\leq 18L/\epsilon$, the curve $\sigma=\sigma_0*\sigma_1*\ldots*\sigma_{m-1}$ has length at most $3L$ and so lies in $\Phi$. We claim that $\gamma$ lies in an $\epsilon$-neighbourhood of $\sigma$. For $t\in[t_i,t_{i+1}]$, the point $\gamma(t)$ is within distance $\epsilon/3$ of $q_i$ by the triangle inequality, since 
    \begin{align*}
        \dist(\gamma(t),q_i)\leq \dist(\gamma(t),\gamma(t_i))+\dist(\gamma(t_i),q_i)\leq \epsilon/3.
    \end{align*}
    Because each $\sigma_i$ has length at most $\epsilon/2$, again using the triangle inequality, we see that $\dist(\gamma(t),\sigma_i(s))\leq 5\epsilon/6<\epsilon$. Therefore $\Phi$ is indeed an $\epsilon$-net.
    \par 
    Lastly, we estimate the number of curves in $\Phi$. Each curve in $\Phi$ is formed by at most $18L/\epsilon$ segments, and hence is uniquely defined by a choice of at most $18L/\epsilon+1$ points in the set $\{p_i\}_{i=1}^N$. Therefore there are fewer than $N^{\frac{18L}{\epsilon}+1}$ curves in $\Phi$, as claimed.
\end{proof}
\begin{figure}
    \includegraphics[width=\textwidth]{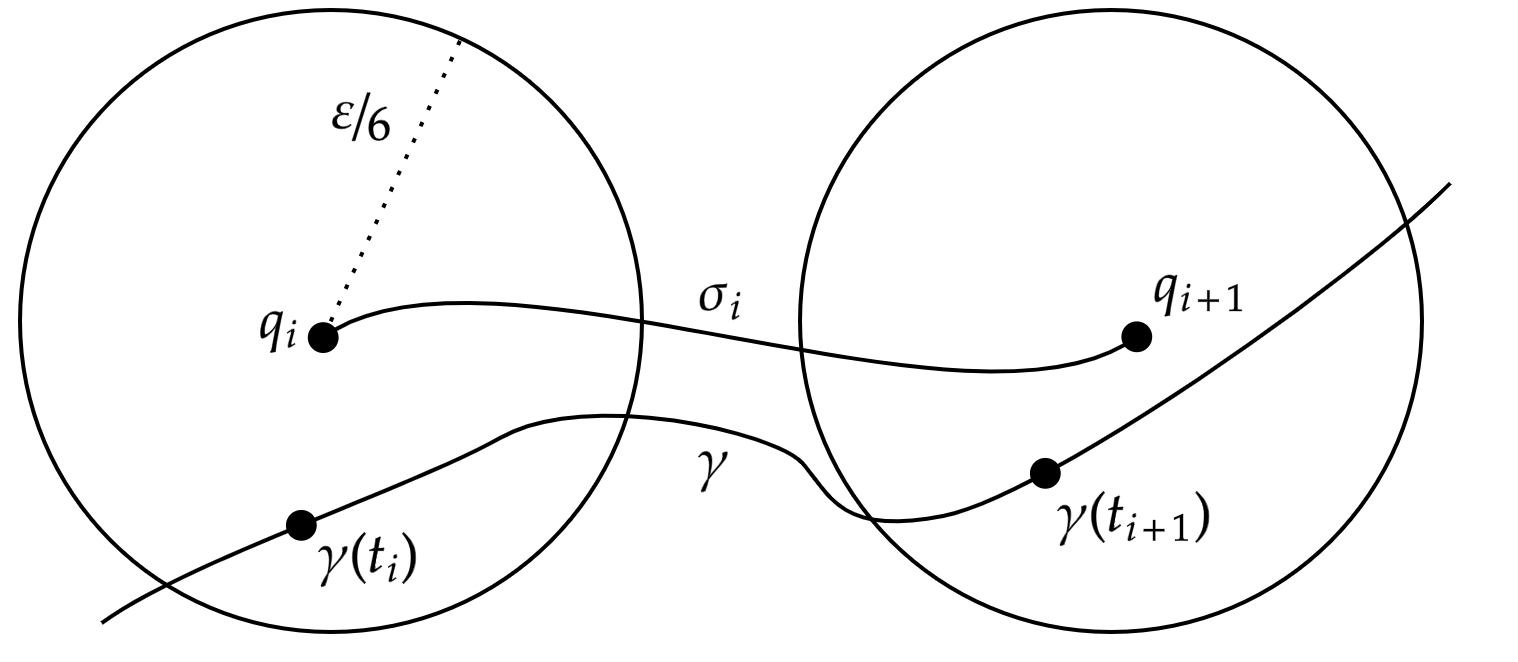}
    \caption{Illustration of the proof of Lemma \ref{lemma:net}.} 
    \label{fig:net}
\end{figure}

Given two sufficiently close curves in $\Lambda M$, we now explain how to find a homotopy between them of short width. In particular, this will allow us to homotope between nearby elements of the net from the previous lemma. We accomplish this in the next lemma, which is an adaptation of Part 2 of Lemma 3.3 due to the fourth author in \cite{rotman2000}. We again include the proof here for the sake of completeness. 

\begin{Lem}\cite{rotman2000}\label{lemma:close_curves_homotopy_bounded_width}
    Pick $a\in \mathbb{R}$ such that $2< a$. Let $\alpha(t), \gamma(t)$ be two closed curves on $M\in \mathscr{M}^D_{k,v}$ such that $\dist(\alpha(t), \gamma(t)) < r(n,k,v,D)/a$. Then there exists a free homotopy between $\alpha(t)$ and $\gamma(t)$ of width at most $4R(n, k, v, D)r(n, k, v, D)/a+2r(n, k, v, D)/a$.
\end{Lem}
\begin{figure}[h]
    \includegraphics[width=\textwidth]{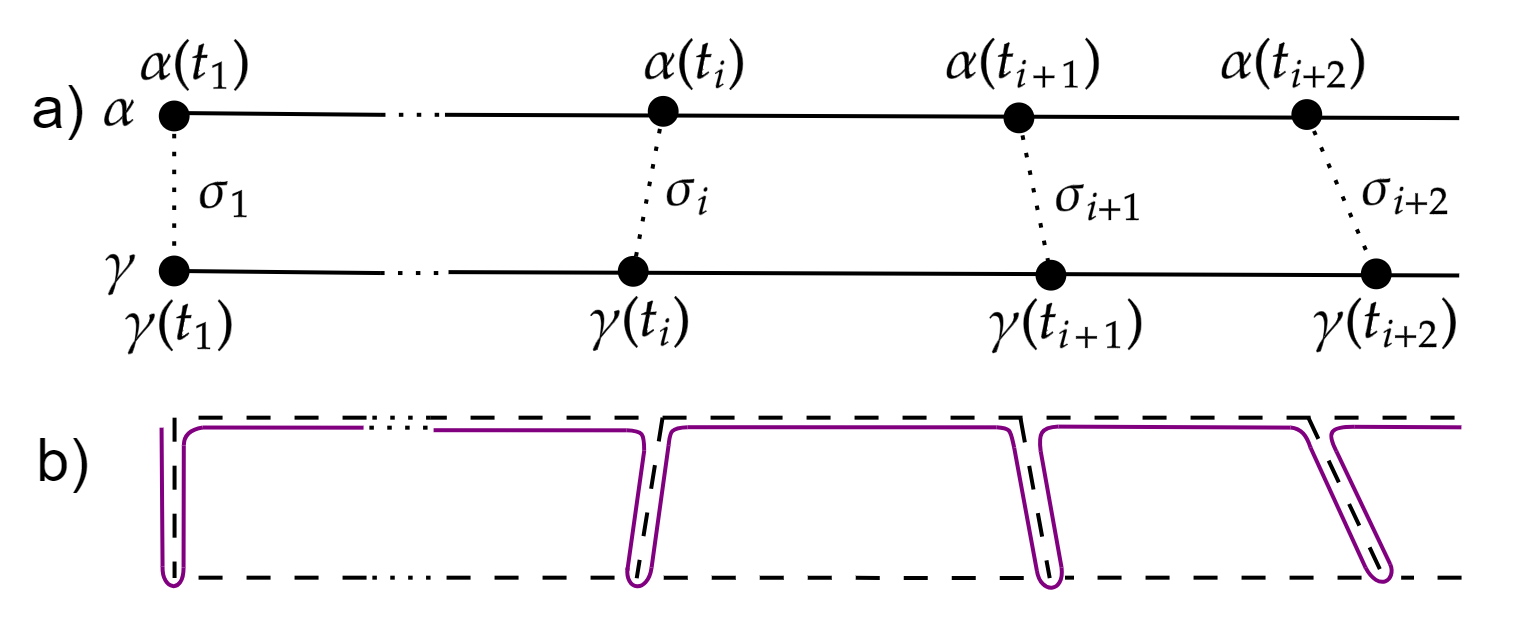}
    \caption{An illustration of (a) the curves $\sigma_i$ and (b) the curve $\widetilde{\alpha}$ in the proof of Lemma \ref{lemma:close_curves_homotopy_bounded_width}.} 
    \label{fig:close_curves_homotopy_bounded_width_ab}
\end{figure}
\begin{proof}
    Since $M\in \mathscr{M}^D_{k,v}$, we may choose  $r=r(n,k,v,D)$ and $R=R(n,k,v,D)$ as in Theorem \ref{lem:grove_petersen}. Reparameterize $\alpha$ and $\gamma$ by the unit interval.
    Subdivide the interval $[0,1]$ by $0=t_0<t_1<\ldots<t_{m-1}<t_{m}=1$, choosing the $t_i$ so that
        \begin{align*}
        \max_i\{\dist(\alpha(t_i),\alpha(t_{i+1})),\dist(\gamma(t_i),\gamma(t_{i+1}))\}<r/a.
    \end{align*}
    Let $\sigma_i$ be a minimizing geodesic connecting $\alpha(t_i)$ to $\gamma(t_i)$ (see Figure \ref{fig:close_curves_homotopy_bounded_width_ab} (a)). Then for every $i$ the closed curve
    $$P_i=-\gamma|_{[t_i,t_{i+1}]}*-\sigma_{i}*\alpha|_{[t_i,t_{i+1}]}*\sigma_{i+1}$$ lies within the closure of $B(\alpha(t_i), 2r/a)$, since it has length at most $4r/a$ and passes through $\alpha(t_i)$.
    \par
    Now, $\alpha$ is homotopic to 
    $$\widetilde{\alpha}=\alpha|_{[t_1,t_{2}]}*\sigma_{2}*-\sigma_{2}*\ldots *\alpha|_{[t_{m-1},t_{m}]}*\sigma_{m}*-\sigma_{m},$$
    via a homotopy of width at most $\max_i\{\ell(\sigma_i)\}< r/a$, simply by extending along $\sigma_{i}|_{[0,t]}*-\sigma_i|_{[0,t]}$ simultaneously for every $i$ (see Figure \ref{fig:close_curves_homotopy_bounded_width_ab} (b)). Similarly, if we define
    $$\Gamma_i=\alpha|_{[t_i,t_{i+1}]}*\sigma_{i+1}*\gamma|_{[t_{i+1},t_{i+2}]}*-\gamma|_{[t_{i+1},t_{i+2}]}*-\sigma_{i+1},$$
    then 
    $\widetilde{\alpha}$ is homotopic to $\Gamma_1*\ldots *\Gamma_{m}$ through a homotopy of width less than $r/a$, since $\dist(\gamma(t_i),\gamma(t_{i+1}))<r/a$ (see Figure \ref{fig:close_curves_homotopy_bounded_width_cd} (a)). However, after reparameterization, $\Gamma_1*\ldots *\Gamma_{m}$ is simply the curve 
    $$\gamma|_{[t_{1},t_{2}]}*P_1*\gamma|_{[t_{2},t_{3}]}*P_2*\gamma|_{[t_{3},t_{4}]}*\ldots * P_m*\gamma|_{[t_{m-1},t_{m}]}.$$
    If  we can contract each loop $P_i$ while fixing its basepoint $\gamma(t_{i+1})$ (see Figure \ref{fig:close_curves_homotopy_bounded_width_cd}(b)), then we have the desired homotopy. Because $P_i$ lies in the closure of the ball of radius $2r/a<r$, by Theorem \ref{lem:grove_petersen} it is contractible inside the ball of radius $R(2r/a)$ through a homotopy of width $R(2r/a)$. Therefore we can contract $P_i$ while fixing the basepoint by conjugating the homotopy with the trajectory of $P_i(0)$, which requires a width of at most $4Rr/a$. Therefore the total width of our homotopy is at most $(4R+2)r/a$. 
\end{proof}
\begin{figure}[h]
    \includegraphics[width=\textwidth]{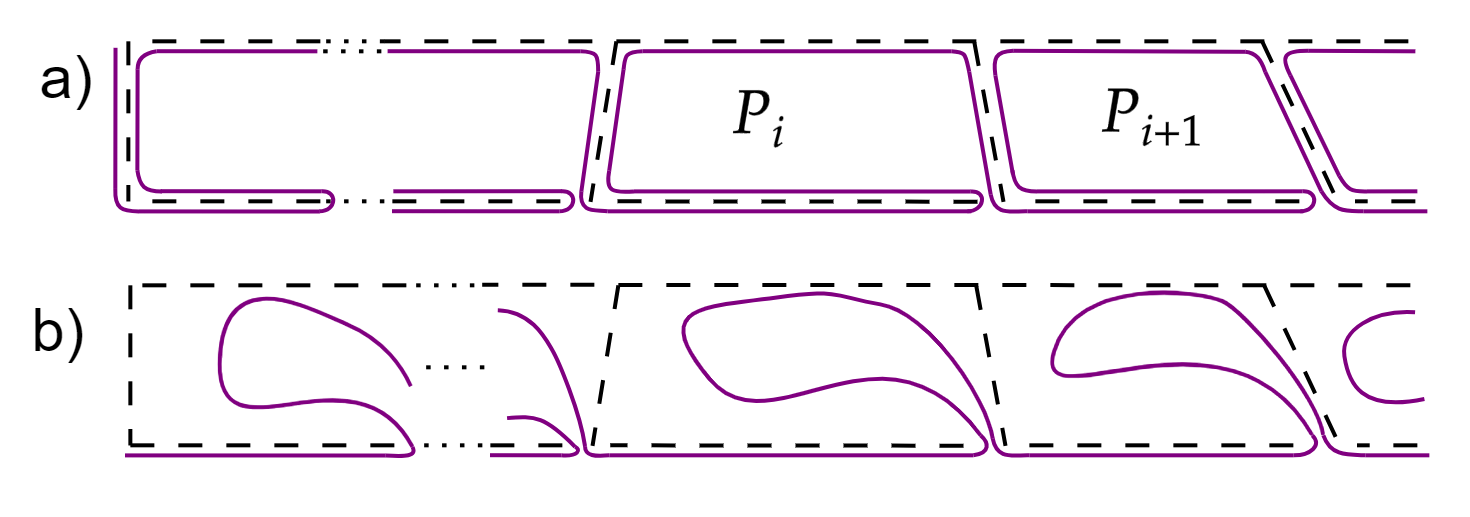}
    \caption{An illustration of (a) the loops $P_i$ and (b) the contraction of each $P_i$ in the proof of Lemma \ref{lemma:close_curves_homotopy_bounded_width}.} 
    \label{fig:close_curves_homotopy_bounded_width_cd}
\end{figure}

\section{Bounding Width and Length}\label{4}

In the previous section, we showed that we can homotope between two sufficiently close curves in $\Lambda M$ through a homotopy of short width. We now wish to convert this into a homotopy through curves that also have short length.

\begin{Lem} \label{lemma:short_hom}
    Let $\gamma(t)$ be a closed piecewise differentiable curve on a closed Riemannian manifold $M$. Suppose that there exists a homotopy $H(t, \tau)=\gamma_{\tau}(t)$ 
    of width $W$ contracting $\gamma(t)=H(t, 0)$ to a point. 
    Then given any $\epsilon>0$, there exists a homotopy, denoted $H'(t, \tau)$, contracting $\gamma$ to a point of width $W$ such that the length of each curve $t\mapsto H'(t, \tau)$ is at most $\ell(\gamma)+3W+\epsilon$.
\end{Lem}
   \begin{figure}[b]
    \includegraphics[width=0.65\textwidth]{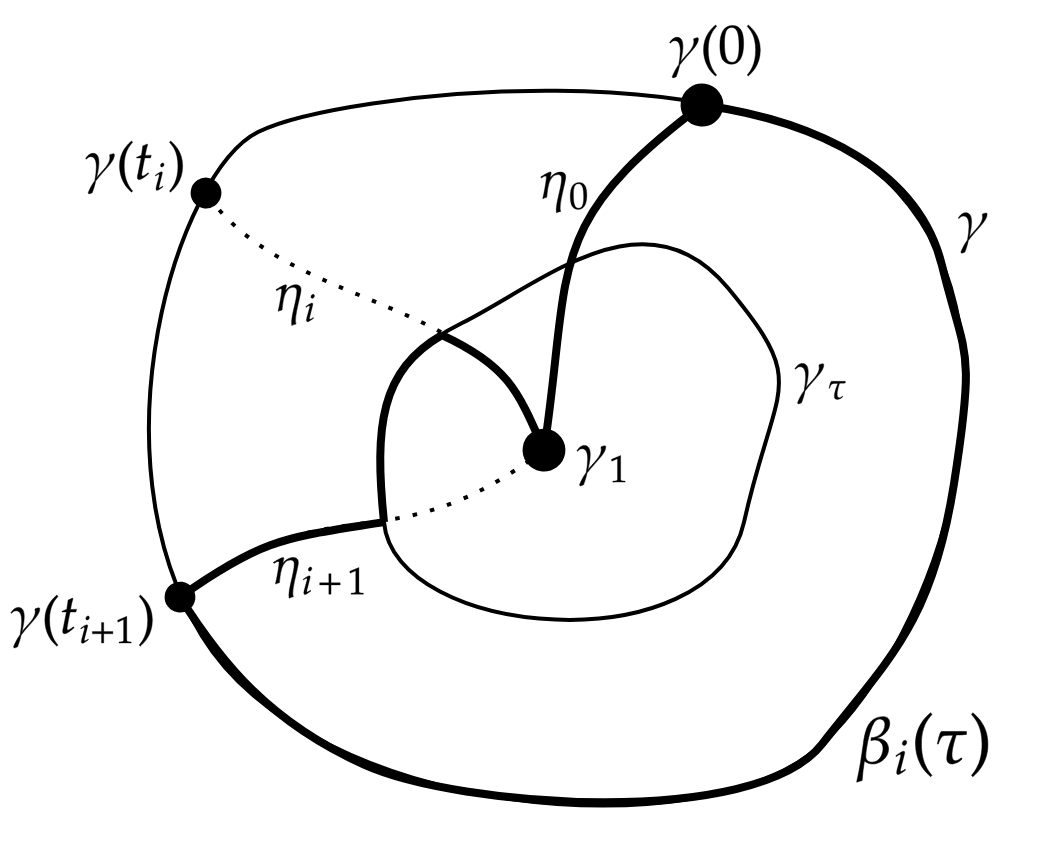}
    \caption{An illustration of the curve $\beta_i$ in the proof of Lemma \ref{lemma:short_hom}.} 
    \label{fig:short_hom}
\end{figure}
\begin{proof}
    Essentially, we will divide $\gamma$ into short segments $\gamma|_{[t_{i-1}, t_{i}]}$ and consider the triangles bounded by $\gamma|_{[t_{i-1}, t_{i}]}$ and the point $\gamma_1$. We then use the homotopy to contract these triangles individually. Pick $m$ sufficiently large and subdivide the interval $[0,1]$ by $0=t_0<t_1<\ldots<t_{m-1}<t_{m}=1$ so that $\ell(\gamma_\tau|_{[t_{i-1}, t_{i}]})<\epsilon$ for all $\tau$. Let $\eta_i(\tau)=\gamma_\tau(t_i)$ be the trajectory of the point $\gamma_0(t_i)$ under our original homotopy, noting that $\eta_0=\eta_{m}$. We define a new homotopy as follows:
    \begin{align*}
        H'(t, \tau)=
        \begin{cases}
            \beta_1(t) & \tau\in [0, t_1]\\
            \vdots & \hskip1cm\vdots\\
            \beta_i(t) & \tau\in[t_{i-1}, t_{i}]\\
            \vdots & \hskip1cm\vdots\\
            \beta_{m}(t) & \tau\in[t_{m-1}, t_{m}]\\
        \end{cases}
    \end{align*}
    where we have
    \begin{align*}
        \beta_1(t) 
        &=\eta_0|_{[0,\tau/t_1]}*\gamma_\tau|_{[0,t_1]}*-\eta_1|_{[0, \tau/t_1]}*\gamma|_{[t_1,1]}\\
        \beta_{m}(t) 
        &=\eta_0|_{[0,f_m(\tau)]}*-\eta_{m}|_{[0,f_m(\tau)]}
    \end{align*}
    and for $0<i<m$ we have
    \begin{align*}
        \beta_i 
        &=\eta_0*-\eta_{i}|_{[0, f_i(\tau)]}*\gamma_\tau|_{[t_{i},t_{i+1}]}*-\eta_{i+1}|_{[0,f_i(\tau)]}*\gamma|_{[t_{i+1},1]},
    \end{align*}
    where $f_i(\tau)=\frac{t_{i}-\tau}{t_{i}-t_{i-1}}$.
    The width of $\tau\mapsto H'(t, \tau)$ is at most $W$, because each point moves no farther than it does under the original homotopy. Moreover, the length of each curve $t\mapsto H'(t,\tau)$ is at most $3W+\ell(\gamma)+\epsilon$, as claimed.
\end{proof}

In the previous lemmas, we have been considering curves in $\Lambda M$, which by definition do not have any specific base point. Because we are ultimately interested in $\Omega_{p,q} M$, that is, curves with endpoints on $p$ and $q$, we will need homotopies that pass through loops with a fixed base point. First, we show how to convert a free homotopy to a homotopy through loops with a fixed basepoint without significantly increasing length.

\begin{Lem} \label{lemma:homotopy_loop_length_bound} 
    Let $\gamma(t)$ be a closed piecewise differentiable curve of length $L$ on a closed Riemannian manifold $M$. Suppose that there exists a homotopy $H(t, \tau)=\gamma_{\tau}(t)$ of width $W$ contracting $H(t, 0)=\gamma(t)$ to a point over curves of length at most $L$.
    Then there exists a homotopy, $H'(t, \tau)$, contracting $\gamma$ to $\gamma(0)$ over loops based at $\gamma(0)$ such that the length of loops in this homotopy is at most $L+2W$.  
\end{Lem}
\begin{figure}[h]
    \includegraphics[width=0.95\textwidth]{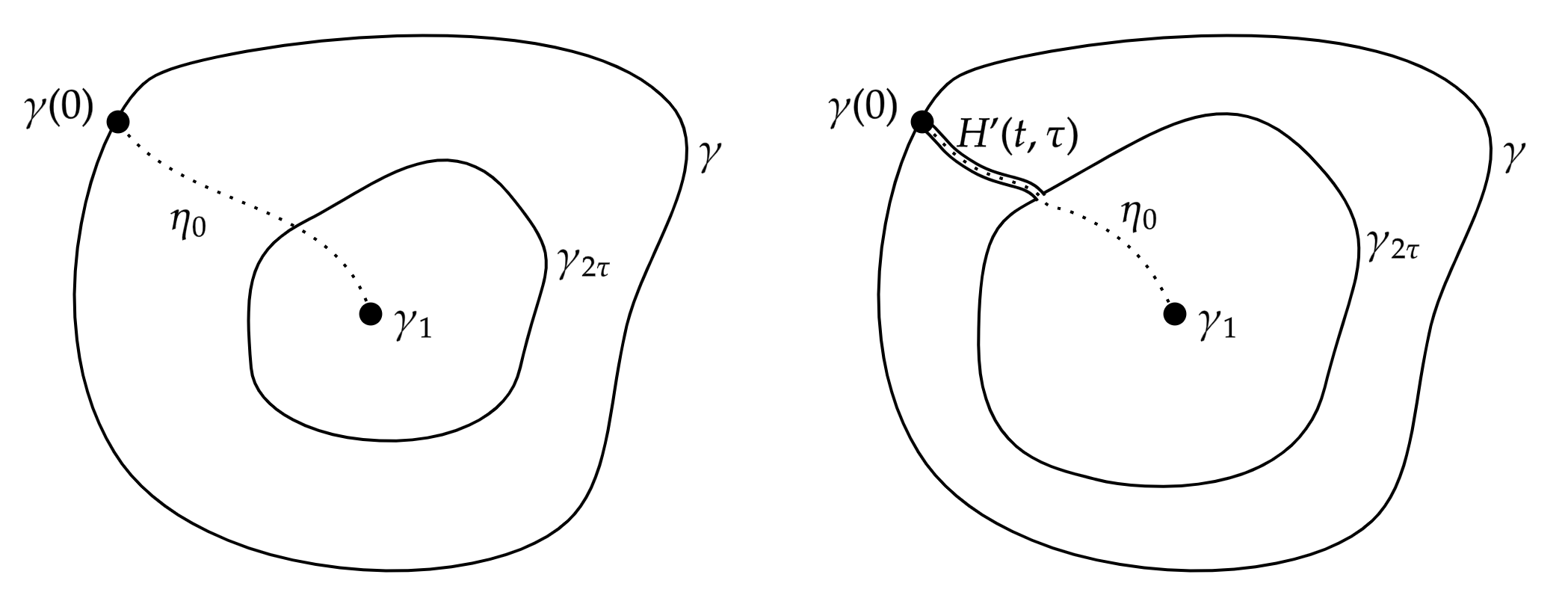}
    \caption{ Illustration of the first half of the homotopy $H'(t, \tau)$ of Lemma \ref{lemma:homotopy_loop_length_bound}.}
    \label{fig:homotopy_loop}
\end{figure}
\begin{proof}
    As in the previous lemma, let $\eta_0(\tau)=\gamma_\tau(0)$ be the trajectory of the point $\gamma(0)$ under the homotopy $\gamma_\tau$. We now define a new homotopy $H'(t, \tau)=\tilde{\gamma}_\tau$ that will contract $\gamma$ to a point, as follows.
    $$
    H'(t, \tau)=
    \begin{cases} 
    \eta_0|_{[0,2\tau]}*\gamma_{2\tau}*-\eta_0|_{[0,2\tau]}, &\tau\in [0, 1/2]
    \\
    \eta_0|_{[0,2-2\tau]}*-\eta_0|_{[0,2-2\tau]}, &\tau\in[1/2, 1].
    \end{cases}
    $$
    Note that the first half of $H'(t, \tau)$ homotopes $\gamma$ to $\eta_0*-\eta_0$ since $\gamma_1$ is a point curve (see Figure \ref{fig:homotopy_loop}). The second half of  $H'(t, \tau)$ homotopes $\eta_0*-\eta_0$ to the point curve $\eta_0(0)=\gamma(0)$.
    The curves in the new homotopy $\tilde{\gamma}_\tau$ have length at most $L+2\ell(\eta)\leq L+2W$ and are all based at $\gamma(0)$, as claimed.
\end{proof}

We now return to the bounded geometry setting and apply the previous lemmas to the $\epsilon$-net we obtained in Lemma \ref{lemma:net}.

\begin{Lem} \label{lemma:homotopy_loop_width_bound}
    Let $M \in \mathscr{M}_{-1,v}^D(n)$ with $\diam(M)=d$. Let $\gamma(t)$ be a piecewise differentiable closed curve on $M$ of length at most $2d$, parametrized on the unit interval. Suppose there is a homotopy $H(t, \tau)=\gamma_\tau$ from $\gamma=\gamma_0$ to a point that passes through curves of length at most $cd$. Then there is a function $W(c,n,v,D)$ such that there exists a homotopy, $H'(t, \tau)$, of width $W(c,n,v,D)$ contracting $\gamma$ to a point. 
\end{Lem}

\begin{proof}
    Since $M \in \mathscr{M}_{-1,v}^D(n)$, we define  $r=r(n,-1,v,D)$ and $R=R(n,-1,v,D)$ as in Theorem \ref{lem:grove_petersen}. Choose $2<a$ and define $\epsilon=r /(4a)$. Recall by the discussion of Theorem \ref{lem:grove_petersen} that given any $\epsilon>0$ we can cover $M$ by the collection of balls $\{B(p_i, \frac{\epsilon}{6})\}_{i=1}^N$ where $N$ is bounded by the function $\psi(\epsilon,D,n)$ by Equation \ref{N}. This is possible due to the assumption that $M$ has curvature bounded below by $-1$. Let $\Phi=\{\xi_i\}_{i=1}^{\widetilde{N}}$ be the $\epsilon$-net covering $\Lambda^{cD} M$ from Lemma \ref{lemma:net} that corresponds to this cover, where we recall that
    \begin{equation*}
        \widetilde{N} \leq N^{18cD/\epsilon + 1} \leq \psi(\epsilon,D,n)^{18cD/\epsilon + 1} .
    \end{equation*}
    \par 
    By assumption, there is a homotopy $\gamma_\tau$ from $\gamma$ to a point that passes through curves of length at most $cd$. We modify $\gamma_\tau$ to produce a new homotopy with bounded width. Subdivide $[0,1]$ into intervals $0=\tau_0 <\tau_1 <\ldots<\tau_m=1$ so that $\dist(\gamma_{\tau_i}, \gamma_{\tau_{i+1}}) \leq r /(2a)$. Pick a closest element in the net $\Phi$ to $\gamma_{\tau_i}$ and call it $\xi_{\tau_i}$. If in this sequence we have $\xi_{\tau_i} = \xi_{\tau_j}$ for some $i< j$, discard the terms $\xi_{\tau_i},\ldots, \xi_{\tau_{j-1}}$, thus obtaining a new sequence in which all elements are different. Thus the number of elements in this sequence will be bounded by the number of elements in $\Phi$, and hence by $\psi(\epsilon,D,n)^{18cD/\epsilon + 1}$. 
    \par
    Since $\dist(\gamma_{\tau_i}, \gamma_{\tau_{i+1}})\leq r /(2a)$ and $\dist(\gamma_{\tau_i}, \xi_{\tau_i})\leq \epsilon$, we have 
    \begin{align*}
        \dist(\xi_{\tau_i}, \xi_{\tau_{i+1}})< r /(2a) + 2\epsilon
        <r/a.
    \end{align*}
    Therefore, by Lemma \ref{lemma:close_curves_homotopy_bounded_width}, there is a homotopy between $\xi_i$ and $\xi_{i+1}$ of width at most $(4R+2)r/a$. Similarly, because $\dist(\gamma_{\tau_i}, \xi_{\tau_i})\leq r/(4a)$ there is a homotopy between $\gamma=\gamma_{\tau_0}$ and $\xi_{\tau_0}$ and a homotopy between $\xi_{\tau_m}$ and the point $\gamma_{\tau_m}$ both with widths at most $(2R+1)r/(2a)$.
    We combine these homotopies to contract $\gamma$ to a point as follows. First, we homotope $\gamma$ to $\xi_{\tau_0}$, then homotope between $\xi_{\tau_i}$ and $\xi_{\tau_{i+1}}$ for all $i$ in sequence. Finally, we homotope $\xi_{\tau_m}$ to the point curve $\gamma_{\tau_m}=\gamma_1$. The width of the total homotopy is bounded by the sum of the individual widths, which is at most 
    \begin{align*}
        W(c,n,v,D) &=2(2R+1)r/(2a)+\widetilde{N}(4R + 2)(r/a)\\
       & \leq
        2(\widetilde{N}+1)(2R + 1)(r/a)\\
        &
        \leq
        2\left(\psi(r/a,D,n)^{18cD/(r/a) + 1}+1\right)(2R + 1)(r/a)
    \end{align*}
    where $\psi$ is defined as in Equation \ref{N}.
   Since this bound holds for any $a>2$ and $r\leq D$, using the estimate from Remark \ref{NR} we obtain 
    \begin{align*}
         \label{eqn:width}
         W(c,n,v,D)
         &<
        \left(\psi(r/a,D,n)^{18cD/(r/a) + 1}+1\right)(2R + 1)D\\
        &\leq 
        \left(\left( 
        \frac{12^n n e^{D(n-1)}2}{(n-1)r^{n}}
        \right)^{32cD/r + 1}+1\right)\left(2c_1(n) \frac{e^{(n-1)D}}{v} + 1\right)D
        \\
        & < e^{ce^{G(n,v,D)}}
  \end{align*}
  for some rational function $G$.
\end{proof}
We now combine the above results to show that a free homotopy through short curves can be converted to a homotopy through short based loops. 
\begin{Lem}
    \label{lemma:homotopy_final_bounds}
    Let $M \in \mathscr{M}_{-1,v}^D(n)$. Let $\gamma(t)$ be a piecewise differentiable closed curve on $M$ of length at most $2d$, where $d$ is the diameter of $M$. Suppose there is a homotopy $\gamma_\tau$ from $\gamma$ to a point that passes through curves of length at most $cd$. 
    Then given any $\epsilon>0$, there exists a homotopy contracting $\gamma$ to $\gamma(0)$ over loops based at $\gamma(0)$ such that the length of any loop in this homotopy is at most $2D+5W(c,n,v,D)+\epsilon$.
\end{Lem}
\begin{proof}
    By Lemma \ref{lemma:homotopy_loop_width_bound}, there is a homotopy of width $W(c,n,v,D)$ contracting $\gamma$ to a point. 
    Thus by Lemma \ref{lemma:short_hom}, given any $\epsilon>0$ there is a homotopy contracting $\gamma$ to a point of width $W(c,n,v,D)$ via curves of length 
    \begin{align*}
        \ell(\gamma)+3W(c,n,v,D)+\epsilon\leq 2D+3W(c,n,v,D)+\epsilon.
    \end{align*}
    Finally, by applying Lemma \ref{lemma:homotopy_loop_length_bound} to this homotopy, we see that there exists a homotopy contracting $\gamma$ to $\gamma(0)$ over loops based at $\gamma(0)$ such that the length of loops in this homotopy is at most $2D+5W(c,n,v,D)+\epsilon$. 
\end{proof}    

\section{Constructing spheres that consist entirely of short loops}\label{5}

In this section, we will complete our proof of Theorem \ref{TheoremAA} by showing that we can homotope a map $h:S^m\to\Omega_{p,q}M$ to a map $g:S^m\to\Omega^L_{pq}M$ for $L=2m(5W(c,n, v, D) +3D)+D+ \delta$. We will divide $S^m$ into cubes and define this map inductively on the $k$-skeleta of these cubes. 
\par
In our proof, we will make use of the following lemma (c.f. Lemma 2.1 in \cite{nr7}).
\begin{Lem} 
    \label{Lemmapathhomotopy}
    Let $\gamma_1, \gamma_2:[0,1] \to M$ be two curves of lengths $l_1, l_2$ respectively, such that $\gamma_1(0)=\gamma_2(0)=p$ and $\gamma_1(1)=\gamma_2(1)=q$. 
    Let $\alpha=-\gamma_1 * \gamma_2$. 
    Suppose $\alpha$ can be contracted to a point 
    via loops 
    based at $p$ of length at most $l_3$. Then there exists a path homotopy between $\gamma_1$ and $\gamma_2$ over curves of length at most 
    $l_3+\min\{l_1,l_2\}$. 
\end{Lem}
\begin{figure}
    \includegraphics[width=0.95\textwidth]{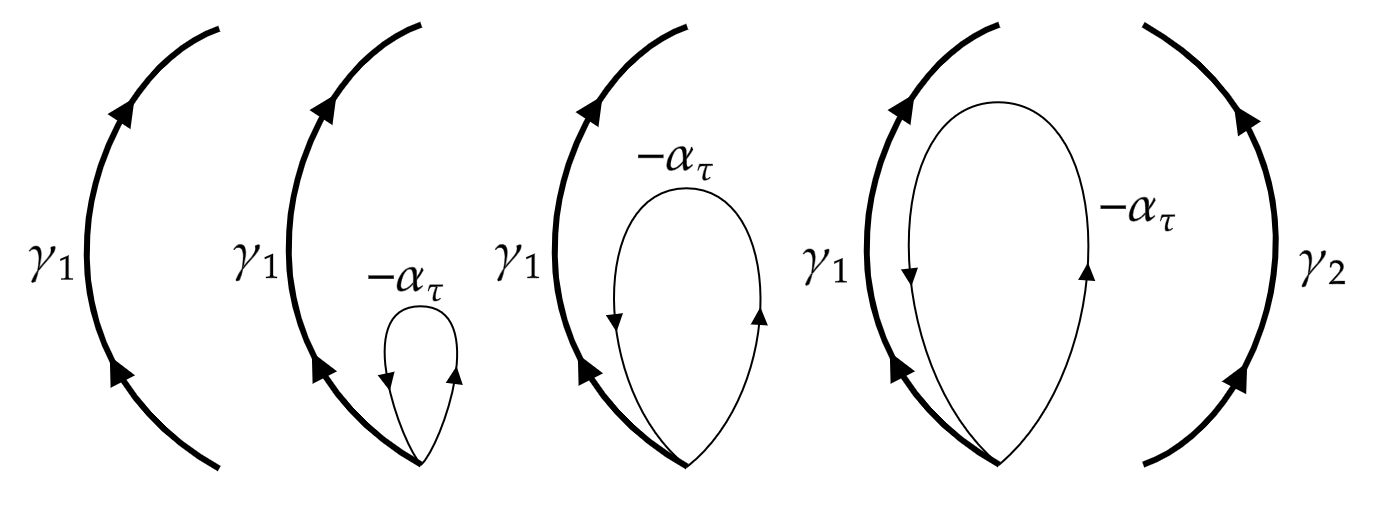}
    \caption{ The homotopy between $\gamma_1$ and $\gamma_2$ in the proof of Lemma \ref{Lemmapathhomotopy}} 
    \label{fig:path-homotopy}
\end{figure}
\begin{proof}
    Let $\alpha_{\tau}$ denote the homotopy that contracts $\alpha$ to a point. Without loss of generality, suppose 
    $l_1\leq l_2$. Otherwise, we can reverse the numbering of the $\gamma_i$ and use the homotopy $-\alpha_{\tau}$. The homotopy between $\gamma_1$ and $\gamma_2$ is first given by $-\gamma_1*\alpha_{1-\tau}$, which interpolates between $\gamma_1$ and $-\gamma_1*\gamma_1*-\gamma_2$. Then we apply the contraction of $\gamma_1*-\gamma_1$ along itself, leaving $\gamma_2$. Since $l_3\geq l_1+l_2$, the longest curve in this homotopy is of length $l_3+l_1=l_3+\min\{l_1,l_2\}$.
\end{proof}

The next lemma will allow us to prove the base case of Theorem \ref{TheoremAA}.
\begin{Lem}
    \label{simplecase} Let $M^n \in \mathscr{M}_{-1,v}^D(n)$ be simply-connected with diameter $d$.
    Let $\alpha:[0,1]\to M$ be a curve connecting points $p, q\in M$. Suppose that any closed curve of length at most $2d$ can be contracted to a point over curves of length at most $cd$. Moreover, suppose that the metric on $M$ is analytic. 
    Then given any $\epsilon>0$, there exists a continuous family of curves $H_{\tau}$, $\tau\in[0,1]$, of length at most $5W(c,n,v,D)+3D+\epsilon$ such that $H_{\tau}(0)=\alpha(0)$  and $H_{\tau}(1)\in \operatorname{Im}(\alpha)$. 
    \par 
    Moreover, consider the function $$\phi:[0,1]\to \operatorname{Im}\alpha$$ given by 
    $    \phi(\tau)=H_\tau(1)    $. This function
    is surjective and non-decreasing in the sense that if $\tau_1\leq \tau_2$ and $\phi(\tau_i)=\alpha(t_i)$, then $t_1\leq t_2$. 
\end{Lem}
\begin{proof}
    First, let us connect $\alpha(0)$ with the points $\alpha(t)$ by minimizing geodesics $\eta_t$ of length at most $d$. Typically, this will not give us a continuous family of curves, as $\alpha$ may intersect the cut locus of $\alpha(0)$. However, if the metric on $M$ is analytic, then there are only finitely many such intersection points. Consider such an intersection point $\alpha(t_0)$. For a sufficiently small $\epsilon>0$, for $t\in[t_0-\epsilon,t_0)$, $\eta_t$ is continuous, well-defined, and $\eta_t$ approaches some minimizing geodesic $\gamma_1$ connecting $\alpha(0)$ and $\alpha(t_0)$ as $t\to t_0^-$. Similarly, for $t\in(t_0,t_0+\epsilon]$, $\eta_t$ approaches some minimizing geodesic $\gamma_2$ as $t\to t_0^+$. By Lemma \ref{lemma:homotopy_final_bounds}, the loop $\gamma_1*-\gamma_2$ is contractible to $\alpha(0)$ over loops based at $\alpha(0)$ of length at most $5W(c,n,v,D)+2D+\epsilon$. 
    Thus, by Lemma \ref{Lemmapathhomotopy}, $\gamma_1$ is path homotopic to $\gamma_2$ over paths of length at most $5W(c,n,v,D)+3D+\epsilon$. 
    This homotopy ``fills in" the digon between $\gamma_1$ and $\gamma_2$ (see Figure \ref{fig:filling}). Note that when we fill in the digon, the point $t_0$ in the domain is replaced by a closed interval, which, in turn, means that  we must reparametrize $H_{\tau}$.
    After filling in the digons at all such points of discontinuity, we obtain the desired family of short curves $H_{\tau}$ that continuously connect the point $\alpha(0)$ with some point $\alpha(t)$. 
    Note that $H_\tau(1)$ will not necessarily equal $\alpha(\tau)$ if we had to fill in a digon, since $H_\tau(1)$ will be constant for the interval of time that parameterizes the family of curves filling the digon. By construction, $\phi(\tau)=H_\tau(1)$ satisfies the required properties.
\end{proof}

\begin{figure}[!htbp]
    \centering
    \includegraphics[width=0.8\textwidth]{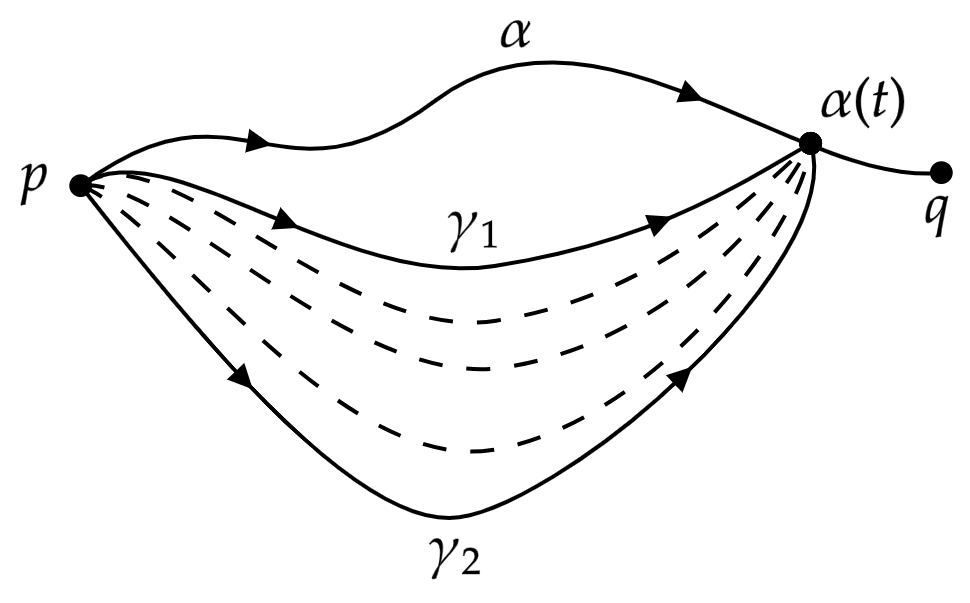}
    \caption{An illustration of the homotopy filling in the digon in the proof of Lemma \ref{simplecase}.}
    \label{fig:filling}
\end{figure}

We now construct the map that will allow us to prove the inductive step by extending a map of the $k$-skeleton to a map of the $(k+1)$-skeleton. For convenience, let $\Omega_{p,\cdot} M$ denote the space of curves in $M$ starting at the point $p$.
\begin{Lem} \label{Lemmashortpaths}
    Suppose that $f: [-\epsilon, \epsilon]^m \to \Omega_{p,q}M$ has the property that for all $x \in [-\epsilon, \epsilon]^m$ and all $t \in [0, 1]$, the path $s\mapsto f_{sx}(t)$ for $s\in[0, \frac{\Vert x \Vert}{\epsilon}]$ has length less than $\delta$, where $\Vert x \Vert$ is the sup norm.
    Consider an assignment 
    $$H_{\partial}: \partial ([-\epsilon, \epsilon]^m)\times [0, 1]\rightarrow \Omega_{p, \cdot}^L(M)$$
    such that each path $t\mapsto H_{\partial}(x, t)$ ends at the point 
    $f_{x}(t)$. Then 
    $H_{\partial}$ can be extended to an assignment 
    $$H:[-\epsilon, \epsilon]^m \times [0,1]\to \Omega_{p, \cdot}(M)$$
    whose image curves $H(x,t)$ start at $p$, end at $f_x(t)$, and have lengths at most $L + 10W(c,n,v,D)+6D+4\delta$.
\end{Lem}
\begin{proof}
    We first describe some curves that we will use in our definition of $H$. Denote the point $(0,\ldots,0)\in [-\epsilon,\epsilon]^m$ by ${\bf 0}$. Consider the curve $f_{\bf0}:[0,1]\to M$. By Lemma \ref{simplecase}, there exists a continuous family of curves $\alpha_t$ such that $\alpha_t(0)=f_{\bf 0}(0)$, $\alpha_t(1)\in\operatorname{Im}f_{\bf 0}$, and the length of each $\alpha_t$ is bounded above by $5W(c,n,v,D)+3D+\delta$ (see Figure ~\ref{fig:steps-ab} (a)). Moreover, Lemma \ref{simplecase} proves that the function $\phi(\tau)=H_\tau(1)$ satisfies $|\phi^{-1}(f_{\bf 0}(t))|=1$ for all but finitely many $t$. For brevity, we will redefine $\alpha_t$ as the set of curves $\alpha_\tau$ for $\tau\in\phi^{-1}(f_{\bf 0}(t))$.
    \par
    Also, for any $x' \in \partial([-\epsilon, \epsilon]^m)$, let $\sigma_{x'}$ be the straight-line path from ${\bf0}$ to $x'$. For each fixed $t \in [0, 1]$, let $\widetilde{\sigma}_{x',t}$ denote the image of $\sigma_{x'}$ under the map $x\mapsto f_x(t)$ (see Figure ~\ref{fig:steps-ab} (b)). Recall that $\ell(\widetilde{\sigma}_{x',t})\leq \delta$ for each $t$ by assumption. Note that the paths $\widetilde{\sigma}_{x',t}$ vary continuously with $x'$ and $t$. 
    \begin{figure}[!htbp]
        \centering
        \includegraphics[width=0.8\textwidth]{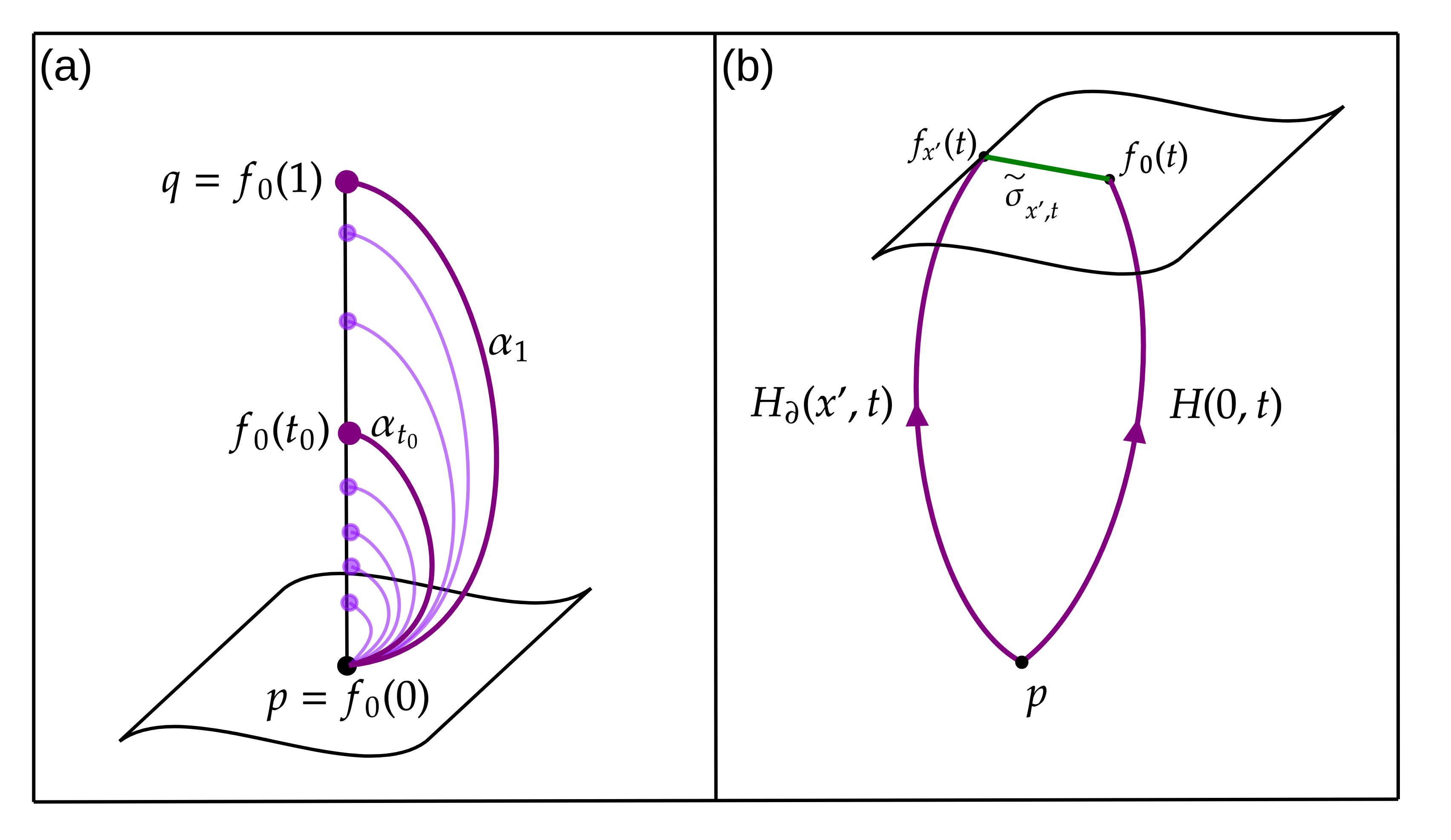}
        \caption{An illustration of (a) the $\alpha_t$ and (b) the $\widetilde{\sigma}_{x', t}$ in the proof of Lemma \ref{Lemmashortpaths}. Note that $\sigma_{x,{\bf0}}$ equals the point curve $p$ for any $x$.}
        \label{fig:steps-ab}
    \end{figure}
    We now describe the construction of $H$ on
    $[-\epsilon, \epsilon]^m \times [0,1]$. Define
    \begin{align*}
        t'(x) = \frac{\Vert x \Vert - \epsilon/4}{\epsilon/2}.
    \end{align*} 
    We then define $H(x,t)$ as
    \begin{align*}
        H(x, t)=
        \begin{cases}        
        H_1(x,t)
        &
        0\leq \Vert x \Vert < \frac{\epsilon}{4}
        \\    
        \\
        H_2(x,t)
        &
        \frac{\epsilon}{4} \leq \Vert x\Vert \leq \frac{3\epsilon}{4}
        \\  
        \\
        H_3(x,t)
        &
        \frac{3\epsilon}{4} <\Vert x \Vert \leq \epsilon
        \end{cases}
    \end{align*}
    where 
    \begin{align*}
        H_1(x,t)&=
        \alpha_{t}
        *\widetilde{\sigma}_{x', t}(s)|_{s\in\left[0,\frac{\Vert x \Vert}{\epsilon}\right]},
        \\
        H_2(x,t)&=
        H_{\partial}(x', t'(x)t)
        *-\widetilde{\sigma}_{x', t'(x)t}
        *-\alpha_{t'(x)t}
        *\alpha_{t}
        *\widetilde{\sigma}_{x', t}(s)|_{s\in\left[0,\frac{\Vert x \Vert}{\epsilon}\right]}, {\mathrm{and}}
        \\
        H_3(x,t)&=
        H_{\partial}(x', t)
        *-\widetilde{\sigma}_{x', t}
        *-\alpha_{t}
        *\alpha_{t}
        *\widetilde{\sigma}_{x', t}(s)|_{s\in\left[0,\frac{\Vert x \Vert}{\epsilon}\right]}.
    \end{align*}
    The three cases above are shown in Parts (a), (b) and (c) of Figure \ref{fig:steps-cdef} respectively.
    Next, in order to make $H$ equal to $H_\partial$ on $\partial ([-\epsilon,\epsilon]^m)\times [0,1]$, we do the following. Let 
    $$\beta_{x',t}=\alpha_t*
    \left( \widetilde{\sigma}_{x', t}(s)\mid_{s = [0,\frac{\Vert x \Vert}{\epsilon}]}\right).$$
    We then extend $H(x,t)$ to $\epsilon <\Vert x \Vert \leq \epsilon+1$ by
    \begin{multline*}
        H(x,t)=H_{\partial}(x', t)
        *-\widetilde{\sigma}_{x', t}(s)\mid_{s\in\left[\frac{\Vert x \Vert}{\epsilon},1\right]}\\
        *-\beta_{x', t}(s)|_{s\in\left[0,1+\epsilon-\Vert x \Vert\right]}
        *\beta_{x', t}(s)|_{s\in\left[0,1+\epsilon-\Vert x \Vert\right]}.
    \end{multline*}
    Note that at $t=1$, we have $H(x,t)=H_\partial(x',t)$, as $\widetilde{\sigma}_{x', 1}$ is a point curve $q$.
    \begin{figure}[!htbp]
        \centering        \includegraphics[width=1\textwidth]{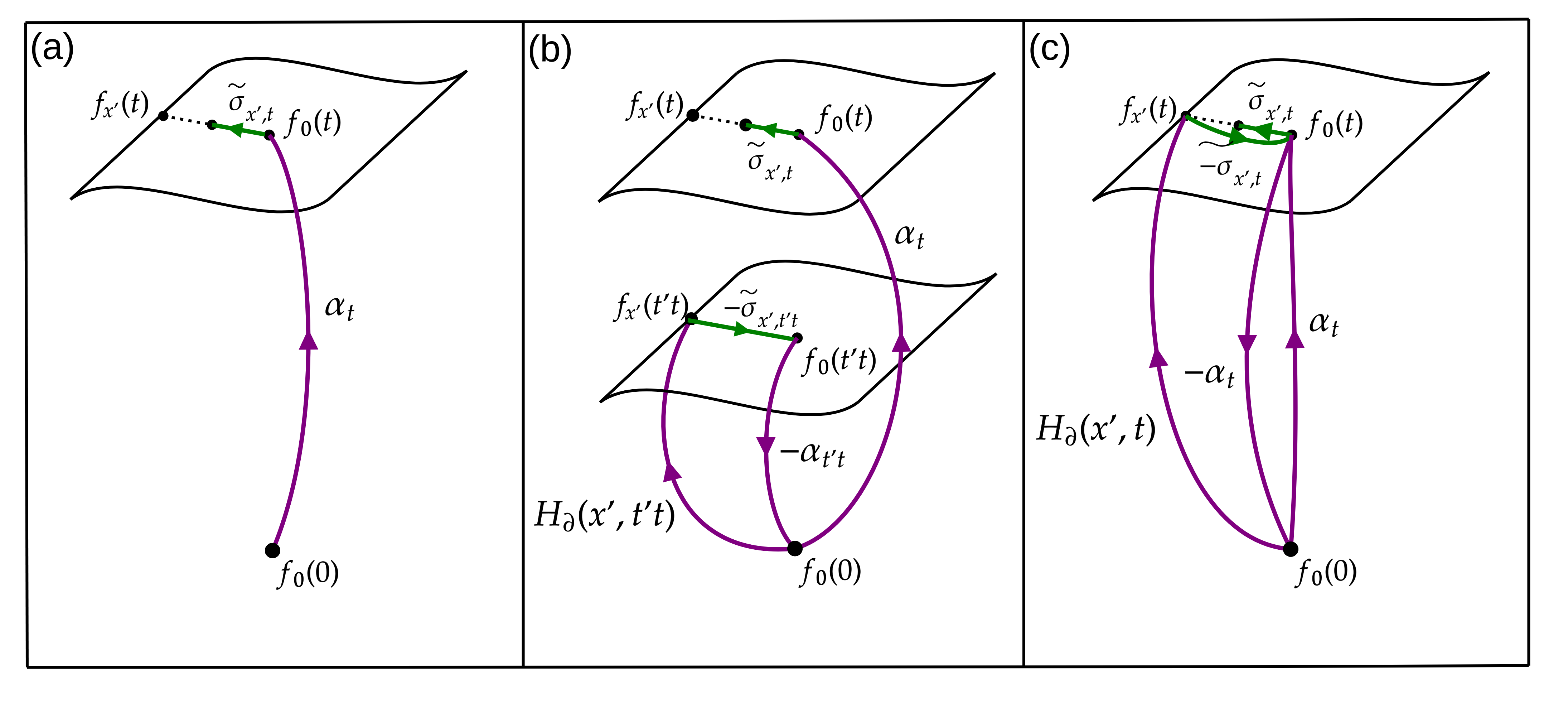}
        \caption{An illustration of the extension of $H_\partial$ to $[-\epsilon, \epsilon]^m\times [0, 1]$ in the proof of Lemma \ref{Lemmashortpaths}.}
        \label{fig:steps-cdef}
    \end{figure}
    Note that when $0\leq \Vert x \Vert < \frac{\epsilon}{4}$, $H(x,t)$ has length at most $2\delta+\ell(\alpha_t)$.
    When $\frac{\epsilon}{4}\leq \Vert x \Vert \leq \epsilon+1$, $H(x,t)$ has length at most $2\delta+L+2\ell(\alpha_t)$.
    Therefore $\ell(H(x,t))\leq 2\delta +L +2(5W(c,n,v,D)+3D+\delta)$, as claimed. 
    Finally, we reparameterize $H$ to have the required domain $[-\epsilon,\epsilon]^m\times [0,1]$. Note that $H(x,t)$ is multivalued when $|\phi^{-1}(f_{\bf{0}}(t))|>1$ or $|\phi^{-1}(f_{\bf{0}}(t'(x)t))|>1$.
    \par
\end{proof}

\section{The Proof of the Main Theorem}\label{sec:main_proof}

We are now in a position to finish the proof of the main  theorem, which we restate for the convenience of the reader.
\begin{Thm}
    Let $M\in \mathscr{M}_{-1,v}^D(n)$ be simply connected and analytic with $\diam(M)=d$. Let $c>0$ such that any closed curve of length at most $2d$ on $M$ can be homotoped to a point over curves of length at most $cd$. Then given any $\delta>0$ and any continuous map $f: S^l \to \Omega_{p,q} M$, there exists  a rational function $G(n, v, D)$ and a map $g:S^l \to \Omega_{p,q}^{L}M$ homotopic to $f$, where 
      $$L\leq 2l(5W(c,n, v, D) +3D)+D+ \delta.$$
\end{Thm}
\begin{figure}[!htbp]
    \centering
      \includegraphics[width=1\textwidth]{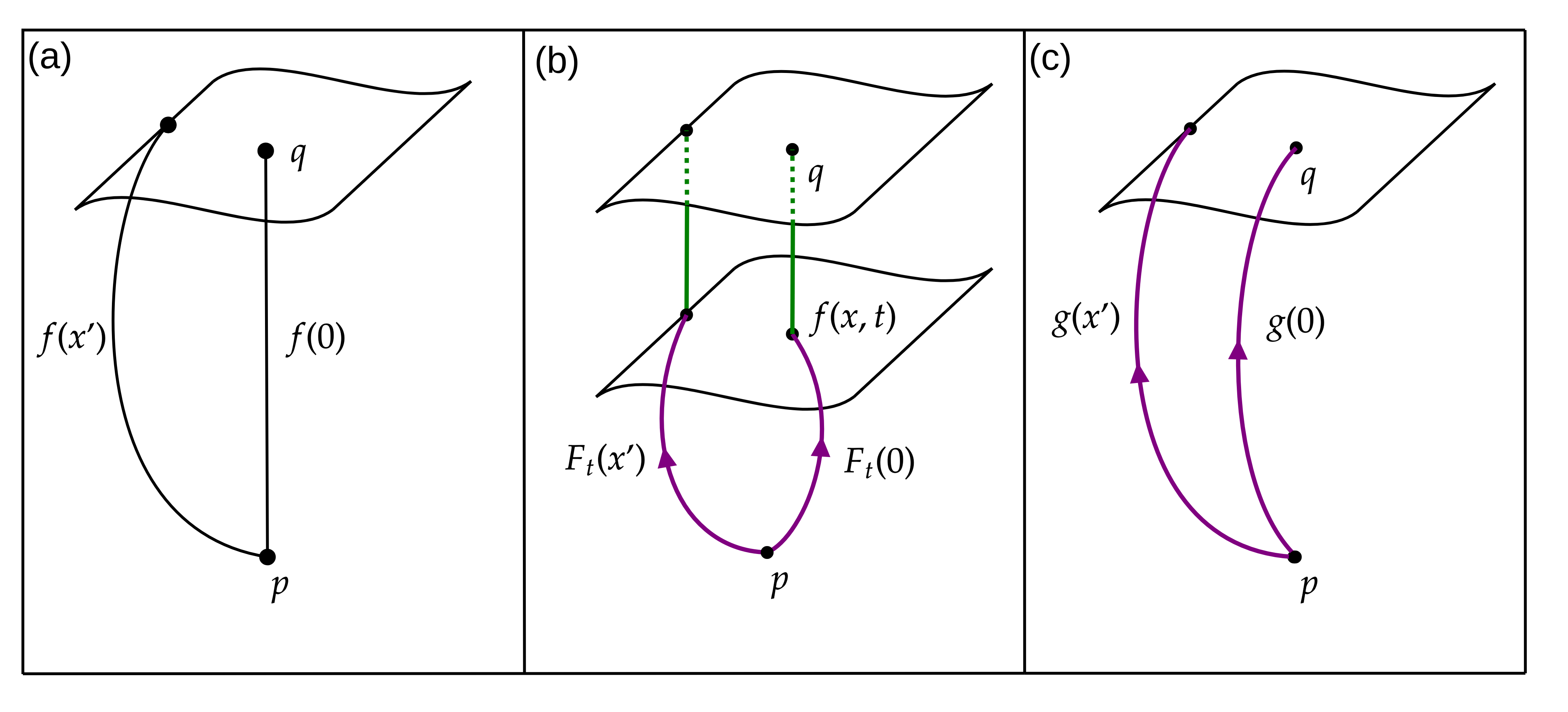}
       \caption{Constructing a homotopy between $h$ and $g$.}
     \label{figure6}
\end{figure}
 
\begin{proof}
Let $f:S^l \longrightarrow \Omega_{p,q}M$ be any continuous map of $S^l$ into the space of piecewise differentiable paths on $M$ between points $p, q \in M$. We must construct a map $g: S^l \to \Omega_{p,q}^LM$, where $L=2l(5W(c,n, v, D) +3D)+D+ \delta$ for any $\delta>0$, as well as a homotopy $F_t$ between $f$ and $g$. 
\par
Fix $\delta >0$. Partition $S^l$ into small $\epsilon$-cubes, $R_i$, such that for any $x, y \in R_i$, we have $\dist(h(x),h(y)) < \delta$. We will construct $g$ and the homotopy $F_t$ from $f$ to $g$ by induction on the skeleta of $S^l$. We start with the 0-skeleton. Given a vertex $v$, define $g(v)$ to be a minimizing geodesic $\alpha_v$ connecting $p$ and $q$. Note that this curve has length at most $D$. Apply Lemma \ref{simplecase} to define a family of curves $H_t$ and a map $\phi$ such that $\phi(t)=H_t(1)\in\operatorname{Im}\alpha_v$, $H_\tau(0)=p$ and $H_t$ has length at most $5W(c,n,v,D)+3D+\epsilon$. Define the multivalued assignment $H_0(v,t):[0,1]\to \Omega_{p,q}M$ by mapping $H_0(v,t)$ to the set of all curves $H_\tau$ with $H_\tau(1)=\alpha(t)$. Moreover, define $F_t(v)=H_t*f(x)|_{[\phi(t),1]}$.
\par
In general, assuming we have extended $g$ and $F_t$ to the $k$-skeleton, we can define a map on the $(k+1)$-skeleton as follows. Note that $f|_{R_i}: [-\epsilon, \epsilon]^m \to \Omega_{p,q}M$ satisfies the condition that $s\mapsto f_{sx}(t)$ has length less than $\delta$. 
Moreover, $H_{k}:\partial([-\epsilon,\epsilon]^{k+1})\times[0,1] \to \Omega^L{p,\cdot}(M)$ satisfies the condition that $F_1(x)=f(x,1)=q$. 
Therefore we can apply Lemma ~\ref{Lemmashortpaths} to define an assignment $H_{k+1}(x,t):[-\epsilon,\epsilon]^{k+1}\times[0,1]\to\Omega_{p,\cdot}M$ that agrees with $F_t$ on $\partial ([-\epsilon,\epsilon]^{k+1})\times[0,1]$. 
We then define $g(x)$ on the $(k+1)$-skeleton as $H_{k+1}(x,1)$, where if $H_{k+1}(x,1)$ is not single-valued, we define $g(x)$ to be the ``final" curve in the family $H_{k+1}(x,1)$. 
Thus $g$ remains a function. We can define $F_t(x):S^m\times[0,1]\to \Omega_{p,q}M$ on the $(k+1)$-skeleton by $F_t(x)=H_{k+1}(x,t)*f(x)|_{[t,1]}$. 
Note that this concatenation makes sense because $H_{k+1}(x,t)$ is a curve joining $p$ and $f(x,t)$ by definition. 
Although $F_t$ may be multivalued, we can convert it to a homotopy through continuous functions. To do so, we recall that at any $x$, the associated curves in the image of $H_{k+1}(x,t)$ can be parameterized as a continuous family by construction. Therefore we can redefine $F_t(x)$ at as this family.
\par 
Each time we apply this lemma, we increase the maximum length of a curve in the image of $g$ by $2(5W(c,n, v, D) +3D +2\delta) +3\delta$. After redefining $\delta$, we obtain a final maximum length of at most $2l(5W(c,n, v, D) +3D)+D+ \delta$. This concludes the proof of Theorem \ref{TheoremAA}.
\end{proof}

\noindent {\bf Acknowledgments.} The authors would like to thank Hannah Alpert and Megan Kerr for many helpful conversations. They are also grateful to the Banff International Research Station Casa Matem\'atica Oaxaca for its support
during the Women in Geometry 2 Workshop (19w5115), where the work on this paper was begun. This article is based upon work supported by the National Science Foundation
under Grant No. DMS-1928930 and the National Security Agency under Grant No. H98230-22-1-0018 while the authors participated in a program hosted by the Mathematical Sciences Research Institute in Berkeley, California, during the summer of 2022. The authors also gratefully acknowledge support from the Princeton IAS Summer Collaborators program in 2024, where the bulk of the work on this project was completed. This material is based upon work supported by the National Science
Foundation under Grant No. DMS-1928930, while Beach, Rotman, and Searle were in
residence at the Simons Laufer Mathematical Sciences Institute (formerly MSRI) in Berkeley, California, during the Fall 2024 semester. Beach was supported in part by an NSERC Canada Graduate Scholarships Doctoral grant.  Contreras Peruyero was partially supported by the UNAM Postdoctoral Program (POSDOC) and CONACyT Research Grant Ciencia de Frontera 2019 CF 217392.  Rotman was partially supported by an NSERC Discovery Grant,  Searle was partially supported by NSF Grants DMS-1906404 and DMS-2204324.

\bibliographystyle{alpha}
\bibliography{references}
\end{document}